\documentclass[a4paper,10pt]{amsart}

\usepackage{amsmath}
\usepackage{amsthm}
\usepackage{amsfonts}
\usepackage{amssymb}
\usepackage{euscript}
\usepackage[all]{xy}

\newcommand{\nc}{\newcommand}

\numberwithin{equation}{section}
\newtheorem{thm}{Theorem}[section]
\newtheorem{prop}[thm]{Proposition}
\newtheorem{lem}[thm]{Lemma}
\newtheorem{cor}[thm]{Corollary}
\theoremstyle{remark}
\newtheorem{rem}[thm]{Remark}

\newtheorem{example}[thm]{Example}

\newtheorem{conj}[thm]{Conjecture}

\nc{\gl}{\mathfrak{gl}}
\nc{\GL}{\mathfrak{GL}}
\nc{\g}{\mathfrak{g}}
\nc{\gh}{\widehat\g}
\nc{\h}{\mathfrak{h}}
\nc{\la}{\lambda}
\nc{\al}{\alpha }
\nc{\ve}{\varepsilon }
\nc{\om}{\omega }

\nc{\ta}{\theta}
\nc{\veps}{\varepsilon}
\nc{\ch}{{\mathop {\rm ch}}}
\nc{\Tr}{{\mathop {\rm Tr}\,}}
\nc{\Id}{{\mathop {\rm Id}}}
\nc{\ad}{{\mathop {\rm ad}}}
\nc{\bra}{\langle}
\nc{\ket}{\rangle}
\nc{\x}{{\bf x}}
\nc{\bs}{{\bf s}}
\nc{\bp}{{\bf p}}
\nc{\bc}{{\bf c}}
\nc{\pa}{\partial}
\nc{\ld}{\ldots}
\nc{\cd}{\cdots}
\nc{\hk}{\hookrightarrow}
\nc{\T}{\otimes}
\newcommand{\bea}{\begin{equation}}
\newcommand{\ena}{\end{equation}}
\nc{\gr}{\mathrm{gr}}
\nc{\ov}{\overline}

\nc{\cO}{\mathcal O}
\nc{\cF}{\mathcal F}
\nc{\cL}{\mathcal L}
\nc{\msl}{\mathfrak{sl}}
\nc{\mgl}{\mathfrak{gl}}
\nc{\U}{\mathrm U}
\nc{\V}{\EuScript V}
\nc{\bH}{\EuScript H}
\nc{\Res}{\mathrm{Res\ }}

\newcommand{\bZ}{{\mathbb Z}}

\newcommand{\dimv}{{\bf dim}}

\begin{document}
\title[Homological approach to the Hernandez-Leclerc construction]
{Homological approach to the Hernandez-Leclerc construction and quiver varieties}
\author{Giovanni Cerulli Irelli, Evgeny Feigin, Markus Reineke}
\address{Giovanni Cerulli Irelli:\newline
Mathematisches Institut, Universit\"at Bonn, Bonn, Germany 53115}
\email{cerulli.math@googlemail.com}
\address{Evgeny Feigin:\newline
Department of Mathematics,\newline National Research University Higher School of Economics,\newline
Russia, 117312, Moscow, Vavilova str. 7\newline
{\it and }\newline
Tamm Department of Theoretical Physics,
Lebedev Physics Institute
}
\email{evgfeig@gmail.com}
\address{Markus Reineke:\newline
Fachbereich C - Mathematik, Bergische Universit\"at Wuppertal, D - 42097 Wuppertal}
\email{reineke@math.uni-wuppertal.de}

\begin{abstract}
In a previous paper the authors have attached to each Dynkin quiver an associative algebra. 
The definition is categorical and the algebra is used to construct  
desingularizations of arbitrary quiver Grassmannians. In the present paper we prove that this algebra is isomorphic to an algebra constructed by Hernandez-Leclerc defined combinatorially and used to describe certain
graded Nakajima quiver varieties. This approach is used to get an explicit realization of the 
orbit closures of representations of Dynkin quivers as affine quotients.  
\end{abstract}

\maketitle

\section{Introduction}
Let $Q$ be a Dynkin quiver and let  $k$ be an algebraically 
closed field of characteristic zero. In \cite{CFR3} the authors have defined an algebra  $B_Q$.
The definition of $B_Q$ is of categorical nature and in principle can be applied not only to the path
algebra of a Dynkin quiver, but to any algebra with finite number of indecomposable representations.
The algebra $B_Q$ plays a
crucial role in the construction of desingularizations of quiver Grassmannians for the quiver $Q$
(see also \cite{CFR1}, \cite{FF}). 
Roughly speaking, given a representation $M$ of the path algebra $kQ$ and a dimension vector ${\bf e}$, 
one can construct a $B_Q$-module $\widehat{M}$ and a dimension vector $\widehat{\bf e}$ of the Gabriel quiver 
of $B_Q$. Then the quiver
Grassmannian ${\rm Gr}_{\bf e}(M)$ can be desingularized by means of the quiver Grassmannian 
${\rm Gr}_{\widehat{\bf e}}({\widehat M})$ of the algebra 
$B_Q$ (if ${\rm Gr}_{\bf e}M$ is not irreducible, then one needs several such quiver Grassmannians).
A central role in the definition of $B_Q$ is played by a certain category of embeddings between projective modules 
of $Q$. This approach allows usage of homological tools for the study of the geometry of various 
objects related to a quiver, such as representation varieties, orbit closures, quiver Grassmannians, etc.
In this paper we continue the study of the algebras $B_Q$. Let us briefly describe our main results.

Let $\widehat{Q}$ be the Gabriel quiver of the algebra $B_Q$. In \cite{CFR3} the quiver $\widehat{Q}$ was
described in terms of the representation theory of $kQ$. For example, $\widehat{Q}$ has two types of
vertices: vertices of the first type are labelled by the non-projective indecomposable $kQ$-modules,
and vertices of the second type are in one-to-one correspondence with the vertices of the initial
quiver $Q$. The arrows of $\widehat{Q}$ can be also explicitly described in representation-theoretic terms. 
By definition, there exists an
admissible ideal $I$ such that  $B_Q=k\widehat{Q}/I$. One of our goals is to describe the ideal $I$.

In \cite[Section 9]{HL}, Hernandez and Leclerc introduced an algebra $\widetilde\Lambda_Q$,
defined by an explicit combinatorial description of its Gabriel quiver $\widetilde\Gamma_Q$ and by an explicit set 
of relations.  This algebra was used to provide    
a realization of certain graded Nakajima quiver varieties
via the representation varieties of the algebra $\widetilde\Lambda_Q$. More precisely, the graded quiver varieties
in question are isomorphic to the representation varieties $R_{\bf d}(Q)$ of the quiver $Q$, and they 
are realized as affine quotients of representation varieties of $\widetilde\Lambda_Q$.      
Our first result is:
\begin{thm}\label{1}
The algebras $B_Q$ and $\widetilde\Lambda_Q$ are isomorphic. In particular, the quiver $\widehat Q$ is isomorphic to 
$\widetilde\Gamma_Q$.
\end{thm} 

We also rederive the Hernandez-Leclerc realization of the representation varieties $R_{\bf d}(Q)$ of the quiver 
$Q$ as affine quotients of representation varieties of $B_Q$. Actually, our approach allows to strengthen 
the results of \cite{HL}. Namely, 
in \cite[Section 9]{HL}, the authors show that not only the representation varieties $R_{\bf d}(Q)$ are isomorphic to certain graded quiver varieties, but moreover the stratification of $R_{\bf d}(Q)$
by the orbits of the structure group $G_{\bf d}$ coincides with the Nakajima stratification. In other words, this means 
that the closure of any $G_{\bf d}$-orbit is isomorphic to the quotient of some representation variety of the 
algebra $\widetilde\Lambda_Q=B_Q$. Our approach allows to describe this representation variety explicitly.
Namely, for a ${\bf d}$-dimensional $kQ$-module $M$ recall the $k\widehat{Q}$-module $\widehat{M}$ 
(used in the construction
of desingularizations of quiver Grassmannians). Let ${\widehat{\bf d}}$ be the dimension vector of $\widehat{M}$.
We prove:
\begin{thm}\label{2} We have the following quotient descriptions of orbit closures:
\begin{enumerate}
\item The closure of the $G_{\bf d}$-orbit of $M$ is isomorphic to an affine quotient of the
variety of $B_Q$-representations of dimension vector ${\widehat{\bf d}}$. 
\item The closure of the $G_{\bf d}$-orbit of $M$ is isomorphic to an affine quotient of the
$G_{\widehat{\bf d}}$-orbit of $\widehat{M}$.
\end{enumerate}
\end{thm}

The paper is organized as follows. In Section \ref{rem} we collect the definition and
main properties of the algebra $B_Q$. In Section \ref{QM} we describe the quotient maps 
relating the representation varieties of $\widehat{Q}$ to that of $Q$ and prove Theorem \ref{2}. 
In Section \ref{HL} the connection with the Hernandez-Leclerc construction is established and 
Theorem \ref{1} is proved.

\section{The algebra $B_Q$} \label{rem}
To a Dynkin quiver $Q$ we attach an associative algebra $B_Q$ (see \cite{CFR3}). The main 
ingredients of the construction are a quiver $\widehat{Q}$, an ideal $I\subset k\widehat{Q}$,
the algebra $B_Q=k\widehat{Q}/I$, and a functor $\Lambda:\bmod\, kQ\rightarrow \bmod\, B_Q$.
All these objects are introduced and described in \cite{CFR3}, except for the ideal $I$, which we
describe in Section \ref{HL}.  Let us recall the definitions.

For the path algebra $kQ$ of $Q$, we consider the category $\bmod\, kQ$ of finite dimensional left modules 
over $kQ$, and we consider the subcategory ${\rm proj}\, kQ$ of projective modules. We define 
${\rm Hom}({\rm proj}\, kQ)$ as the $k$-linear category whose objects are morphisms $P\rightarrow Q$ 
between objects of ${\rm proj}\, kQ$ and whose morphism spaces from $P\rightarrow Q$ to $R\rightarrow S$ 
are commutative squares
$$\begin{array}{ccc}P&\rightarrow&Q\\ \downarrow&&\downarrow\\ R&\rightarrow&S.\end{array}$$
Let $\mathcal{H}_Q$ be the full subcategory of ${\rm Hom}({\rm proj}\, kQ)$ with objects being injective 
morphisms $P\rightarrow Q$ such that the image is not contained in any proper direct summand of $Q$. We define 
the algebra $B_Q$ as ${\rm End}(A)^{\rm op}$ for a minimal additive generator $A$ in $\mathcal{H}_Q$. 
More precisely, the indecomposable objects in $\mathcal{H}_Q$ are of two types, namely:
\begin{itemize}
\item for every non-projective indecomposable object $U$ of $\bmod\, kQ$, an embedding $P_U\subset Q_U$ such 
that $0\rightarrow P_U\rightarrow Q_U\rightarrow U\rightarrow 0$ is a minimal projective resolution,
\item for every vertex $i$ of $Q$, the trivial embedding $P_i=P_i$,
\end{itemize}
where $P_i$ is the projective cover of the simple $kQ$-module $S_i$, attached to the vertex $i$.

Thus the algebra $B_Q$ is given as 
$$B_Q={\rm End}_{\mathcal{H}_Q}(\bigoplus_U(P_U\subset Q_U)\oplus\bigoplus_i(P_i=P_i))^{\rm op}.$$

Then the quiver $\widehat{Q}$ of $B_Q$ is given as follows:
\begin{itemize}
\item there are vertices
\begin{itemize}
\item $[U]$ for every non-projective indecomposable in $\bmod\, kQ$,
\item $[i]$ for every vertex of $Q$.
\end{itemize}
\item There are arrows
\begin{itemize}
\item $[U]\rightarrow[V]$ whenever there is an irreducible map $V\rightarrow U$ in $\bmod\, kQ$,
\item $[i]\rightarrow[S_i]$ for every vertex $i$ of $Q$ which is not a sink,
\item $[\tau^{-1} S_i]\rightarrow[i]$ for every vertex $i$ of $Q$ which is not a source.
\end{itemize}
\end{itemize}

Thus $B_Q=k\widehat{Q}/I$ for a certain admissible ideal $I$, and we think of left modules over $B_Q$ as 
certain representations of $\widehat{Q}$. Moreover, the category $\bmod\, B_Q$ of left modules over $B_Q$ 
is equivalent to the category $\bmod\,\mathcal{H}_Q^{\rm op}$ of contravariant $k$-linear functors from $\mathcal{H}_Q$ to the category $\bmod\, k$ of finite dimensional $k$-vector spaces.

We have a restriction functor ${\rm res}$ from $\bmod\, B_Q$ to $\bmod\, kQ$ given on the level of functors 
as follows: we restrict a functor $F:\mathcal{H}_Q^{\rm op}\rightarrow\bmod\, k$ to the subcategory of
$\mathcal{H}_Q$ of objects of the form $(P=P)$ yielding a functor 
$F:({\rm proj}\, kQ)^{\rm op}\rightarrow\bmod\, k$. The category of such functors is naturally 
equivalent to $\bmod\, kQ$.

We give a description of the functor ${\rm res}$ on the level of representations: given a representation 
$F$ of $\bmod\, B_Q$, it can be viewed as a special kind of representation of the quiver $\widehat{Q}$. We define 
the representation $M={\rm res}\, F$ of $Q$ as follows: first, we define $M_i=F_{[i]}$ for all vertices $i$ 
of $Q$. Let $\alpha:i\rightarrow j$ be an arrow of $Q$. Then ${\rm Ext}^1(S_i,S_j)$ is non-zero (and in 
fact one-dimensional), thus there exists a non-zero morphism $\tau^{-1}S_j\rightarrow S_i$ in $\bmod\, kQ$. 
Factoring it into irreducibles, there thus exists a path $\tau^{-1}S_j\rightarrow\ldots\rightarrow S_i$ 
in the Auslander-Reiten quiver of $kQ$. By the definition of $\widehat{Q}$, this yields a path
$$[i]\rightarrow[S_i]\rightarrow\ldots\rightarrow[\tau^{-1}S_j]\rightarrow[j]$$
in $\widehat{Q}$. We define $M_\alpha:M_i\rightarrow M_j$ as the composition of the maps in the representation 
$F$ of $\widehat{Q}$ corresponding to this path. Well-definedness, that is, independence of the path, 
of this definition follows from the fact that $F$ is a representation of $B_Q$, and that there is a 
one-dimensional space of morphisms from $(P_j=P_j)$ to $(P_i=P_i)$ in $\mathcal{H}_Q$.

We define a functor $\Lambda:\bmod\, kQ\rightarrow\bmod\, B_Q$ as follows: $\Lambda(M)=\widehat{M}$, 
viewed as a contravariant functor on $\mathcal{H}_Q$, is given by 
$$\widehat{M}(P\rightarrow Q)={\rm Im}({\rm Hom}(Q,M)\rightarrow{\rm Hom}(P,M)).$$

On the level of representations, we can view $\widehat{M}$ as a representation of $\widehat{Q}$ as follows:
$\widehat{M}_{[i]}=M_i$ and $\widehat{M}_{[U]}={\rm Im}({\rm Hom}(Q_U,M)\rightarrow{\rm Hom}(P_U,M))$ 
for a minimal projective resolution $0\rightarrow P_U\rightarrow Q_U\rightarrow U\rightarrow 0$ of $U$ in $\bmod\,Q$.

\section{Quotient map}\label{QM}
As before, let $Q$ be a Dynkin quiver with set of vertices $Q_0$ and arrows $\alpha:i\rightarrow j$, and let $k$ be an 
algebraically closed field of characteristic zero. For a dimension vector ${\bf d}\in{\bf N}Q_0$, 
we fix vector spaces $M_i$ of dimension $d_i$ for $i\in Q_0$ and define 
$R_{\bf d}(Q)=\bigoplus_{\alpha:i\rightarrow j}{\rm Hom}_{\bmod\,k}(M_i,M_j)$, 
on which the group $G_{\bf d}=\prod_{i\in Q_0}{\rm GL}(M_i)$ acts via base change 
$(g_i)_i(M_\alpha)_\alpha=(g_jM_\alpha g_i^{-1})_{\alpha:i\rightarrow j}$. 
The orbits $\mathcal{O}_M$ for  this action correspond bijectively to the isomorphism classes 
$[M]$ of representation of $kQ$  of dimension vector ${\bf d}$ by definition. We are interested in the 
geometry of the Zariski orbit closures $\overline{\mathcal{O}_M}$.

We write $\widehat{\bf d}$ for the dimension vector of $\widehat{M}=\Lambda(M)$ as a representation of 
$\widehat{Q}$. In particular, $\widehat{\bf d}$ has entries $\widehat{d}_U$ for every non-projective 
indecomposable $U$ of $kQ$ and entries $\widehat{d}_i$ (which coincide with the $d_i$ by definition of $\Lambda$).
We can consider the variety of representations $R_{\widehat{\bf d}}(\widehat{Q})$, which contains the closed subvariety $R_{\widehat{\bf d}}(B_Q)$ of representations of $B_Q$, that is, those which are annihilated by the ideal $I$.

On $R_{\widehat{\bf d}}(B_Q)$ we have an action of the structure group $G_{\widehat{\bf d}}$, which can be
canonically written as the product of a group $G_{\widehat{\bf d}}'$ and the group $G_{\bf d}$, where the 
first subgroup consists of the structure groups at the vertices $[U]$ of $\widehat{Q}$, and the second 
subgroup consists of the structure groups at the vertices $[i]$.

The restriction functor ${\rm res}:\bmod\, B_Q\rightarrow\bmod\, kQ$ of \cite{CFR3} induces a $G_{\bf d}$ -equivariant map $\pi$ from $R_{\widehat{\bf d}}(B_Q)$ to $R_{\bf d}(Q)$.

\begin{prop} The image of the induced map $\pi^*:k[R_{\bf d}(Q)]\rightarrow k[R_{\widehat{\bf d}}(B_Q)]$ between coordinate rings coincides with the ring $k[R_{\widehat{\bf d}}(B_Q)]^{G_{\widehat{\bf d}}'}$ of $G_{\widehat{\bf d}}'$-invariant functions.
\end{prop}
\begin{proof} 
The following method is the deframing procedure of \cite{CB}. We introduce an auxilliary quiver 
$\tilde{Q}$ associated to ${\bf d}$ as follows:
\begin{itemize}
\item $\tilde{Q}$ has vertices
\begin{itemize}
\item $[U]$ for the non-projective indecomposables $U$ of $\bmod\, kQ$,
\item and one additional vertex $\infty$.
\end{itemize}
\item We have the following arrows in $\tilde{Q}$:
\begin{itemize}
\item arrows $[U]\rightarrow [V]$ corresponding to the irreducible maps $V\rightarrow U$ in $\bmod\, kQ$,
\item $d_i$ arrows from $\infty$ to $[S_i]$ for every vertex $i$ of $Q$ which is not a sink,
\item $d_i$ arrows from $[\tau^{-1}S_i]$ to $\infty$ for every vertex $i$ of $Q$ which is not a source.
\end{itemize}
\end{itemize}
We define a dimension vector $\tilde{\bf d}$ for $\tilde{Q}$ by $\tilde{\bf d}_{[U]}=\widehat{d}_{[U]}$ and $\tilde{\bf d}_\infty=1$.

A choice of bases of the spaces $M_i$ yields an isomorphism between $R_{\widehat{\bf d}}(\widehat{Q})$ and
$R_{\tilde{\bf d}}(\tilde{Q})$. Noting that the structure group for the latter variety of representations 
is isomorphic to $G_{\widehat{\bf d}}'\times{\bf G}_m$, this isomorphism is also 
$G_{\widehat{\bf d}}'$-equivariant. In particular, we have
$$k[R_{\widehat{\bf d}}(\widehat{Q})]^{G_{\widehat{\bf d}}'}\simeq 
k[R_{\tilde{\bf d}}(\tilde{Q})]^{G_{\widehat{\bf d}}'\times{\bf G}_m}=
k[R_{\tilde{\bf d}}(\tilde{Q})]^{G_{\widehat{\bf d}}'},$$
since the dilation action at the one-dimension space can be neglected. 
We now apply the main result of \cite{LBP} that the latter invariant ring is generated by taking traces 
along oriented cycles. By cyclic invariance of the trace, noting that the subquiver supported outside the 
vertex $\infty$ has no oriented cycles, we only need to consider traces along cycles starting and ending in 
the vertex $\infty$. Translating these invariant functions back to $R_{\widehat{\bf d}}(\widehat{Q})$, we see 
that the invariant ring $k[R_{\widehat{\bf d}}(\widehat{Q})]^{G_{\widehat{\bf d}}'}$ is generated by the 
matrix entries representing paths in $\widehat{Q}$ starting and ending in some of the vertices $[i]$. Without 
loss of generality, we can restrict to the matrix entries representing paths in $\widehat{Q}$ starting in $[i]$ 
and ending in $[j]$ whenever we have an arrow $i\rightarrow j$ in $Q$. But the map $\pi$ is induced from the 
functor ${\rm res}$ which precisely defines the matrices representing the arrows $i\rightarrow j$ by these 
paths in $\widehat{Q}$. The proposition is proved.
\end{proof}

\begin{rem} The above argument replaces reference to the theory of Nakajima quiver varieties in \cite[Proof of Proposition 9.4]{HL}.
\end{rem}



\begin{thm} 
Via the map $\pi$, the variety $\overline{\mathcal{O}_M}$ is the quotient of 
$\overline{\mathcal{O}_{\widehat{M}}}$ by $G_{\widehat{d}}'$.
\end{thm}
\begin{proof} 
By the previous proposition, the map 
$\pi:R_{\widehat{\bf d}}(B_Q)\rightarrow\pi(R_{\widehat{\bf d}}(B_Q))\subset R_{\bf d}(Q)$ is a 
quotient by $G_{\widehat{\bf d}}'$. The subset $\overline{\mathcal{O}_{\widehat{M}}}$ is closed and 
$G_{\widehat{\bf d}}'$-invariant, thus the restriction 
$\pi:\overline{\mathcal{O}_{\widehat{M}}}\rightarrow
\pi(\overline{\mathcal{O}_{\widehat{M}}})$ is a quotient. The latter image is closed by properties of quotients, it is $G_{\bf d}$-stable by $G_{\bf d}$-equivariance of $\pi$, and it contains $\pi(\mathcal{O}_{\widehat{M}})=\mathcal{O}_M$ as a dense subset. Thus it coincides with $\overline{\mathcal{O}_M}$. Again by properties of quotients, 
the restriction is a quotient itself.
\end{proof}

We will show that, in fact, the whole representation variety $R_{\widehat{\bf d}}(B_Q)$ is mapped to the closure of $\mathcal{O}_M$ under the map $\pi$. To do this, we have to know the precise relation between the functors $F$ and $\widehat{{\rm res}\, F}$, which we describe using the methods of \cite{CFR3}:

\begin{thm} For every functor $F$ in $\bmod\,\mathcal{H}_Q^{\rm op}$, there exist canonical exact sequences
$$0\rightarrow F_1\rightarrow F\rightarrow{\rm Ext}_{B_Q}^1(\widehat{{\rm Coker}\,\_},F)\rightarrow 0,$$
$$0\rightarrow{\rm Ext}_{B_Q}^1(F_1,\widehat{\tau_Q{\rm Coker}\,\_})^*\rightarrow F_1\rightarrow\widehat{{\rm res}\, F}\rightarrow 0,$$
and dually
$$0\rightarrow{\rm Ext}_{B_Q}^1(F,\widehat{\tau_Q{\rm Coker}\,\_})^*\rightarrow F\rightarrow F_2\rightarrow 0,$$
$$0\rightarrow \widehat{{\rm res}\, F}\rightarrow F_2\rightarrow{\rm Ext}_{B_Q}^1(\widehat{{\rm Coker}\,\_},F_2)\rightarrow 0.$$
\end{thm}

\begin{proof} For every functor $F$, we define functors $F_1$, $F_2$ and $F_3$ as follows: given an object $P\subset Q$ of $\mathcal{H}_Q$, we have a canonical sequence of maps $$\xymatrix@1@C=15pt{(P=P)\ar^{f}[r]&(P\subset Q)\ar^g[r]&(Q=Q)}$$ in $\mathcal{H}_Q$ by \cite[Proof of Lemma 5.3]{CFR3}. We define
$$
\begin{array}{lcl}
F_1(P\subset Q)&=&{\rm Im}(\xymatrix@1@C=25pt{F(Q=Q)\ar^{F(g)}[r]& F(P\subset Q)}),\\
F_2(P\subset Q)&=&{\rm Im}(\xymatrix@1@C=25pt{F(P\subset Q)\ar^{F(f)}[r]&F(P=P)}),\\
F_3(P\subset Q)&=&{\rm Im}(\xymatrix@1@C=25pt{F(Q=Q)\ar^{F(g\circ f)}[r]&F(P=P)}),\\
\end{array}
$$
with the natural definition on morphisms. In particular $F_i(P=P)=F(P=P)$, for every $i=1,2,3$. 
We thus have
$$F_3=(F_1)_2=(F_2)_1=\widehat{{\rm res}\, F}$$
by definition. Thus the existence of the second and fourth claimed exact sequence follows from existence of the first and third. Existence of the first sequence is equivalent to exactness of
$$0\rightarrow {\rm Im}(F(Q=Q)\rightarrow F(P\subset Q))\rightarrow F(P\subset Q)\rightarrow {\rm Ext}^1(\widehat{Q/P},F)\rightarrow 0$$
for all objects $P\subset Q$ of $\mathcal{H}_Q$. By \cite[Proof of Theorem 5.6]{CFR3}, we have a projective resolution of functors
$$0\rightarrow{\rm Hom}(\_,(P\subset Q))\rightarrow{\rm Hom}(\_,(Q=Q))\rightarrow\widehat{Q/P}\rightarrow 0.$$
Applying ${\rm Hom}(\_,F)$ to this sequence and using Yoneda's Lemma, this yields a right exact sequence
$$F(Q=Q)\rightarrow F(P\subset Q)\rightarrow{\rm Ext}^1(\widehat{Q/P},F)\rightarrow 0.$$
This proves the above exactness claim.

To prove the existence of the third exact sequence, we use methods of Auslander-Reiten theory, to which we refer to \cite{ASS}. We denote by $\tau=\tau_Q$ the Auslander-Reiten translate in the category $\bmod\, kQ$, and by $\tau_{B_Q}$ the Auslander-Reiten translate in the category $\bmod\, B_Q$.

As above, we have to prove exactness of
$$0\rightarrow{\rm Ext}^1(F,\widehat{\tau(Q/P)})^*\rightarrow F(P\subset Q)\rightarrow{\rm Im}(F(P\subset Q)\rightarrow F(P=P))\rightarrow 0.$$
Since the image of $\Lambda$ consists of objects of injective dimension at most one \cite[Theorem 5.6]{CFR3}, we can apply the Auslander-Reiten formula \cite[Corollary 2.15]{ASS} and identify $${\rm Ext}^1(F,\widehat{\tau(Q/P)})^*\simeq {\rm Hom}(\tau_{B_Q}^{-1}\widehat{\tau(Q/P)},F).$$
We use the (inverse) Nakayama functors $\nu^{(-)}$ (resp. $\nu^{(-)}_{B_Q}$) of the category $\bmod\, kQ$ (resp. $\bmod\, B_Q$), see \cite[IV.2]{ASS}. The functor $\nu_{B_Q}^-$ is given, by definition, by $\nu^-_{B_Q}({\rm Hom}((P\subset Q),\_)^*)={\rm Hom}(\_,(P\subset Q))$. Starting from an object $P\subset Q$ of $\mathcal{H}_Q$, we consider the exact sequence $0\rightarrow P\rightarrow Q\rightarrow Q/P\rightarrow 0$, which induces an exact sequence $$0\rightarrow\tau(Q/P)\rightarrow \nu P\rightarrow \nu Q\rightarrow\nu (Q/P)=0$$
by \cite[Proposition 2.4.(a)]{ASS}; the last equality follows since $Q/P$ has no projective direct summand by assumption. By \cite[Proof of Theorem 5.6]{CFR3}, this injective coresolution of $\tau (Q/P)$ induces an injective coresolution
$$0\rightarrow \widehat{\tau(Q/P)}\rightarrow{\rm Hom}((P=P),\_)^*\rightarrow{\rm Hom}((P\subset Q),\_)^*\rightarrow 0$$
of $\widehat{\tau(Q/P)}$. This in turn, by \cite[Proposition 2.4.(b)]{ASS}, yields a right exact sequence
$$\underbrace{\nu_{B_Q}^-{\rm Hom}((P=P),\_)^*}_{\simeq{\rm Hom}(\_,(P=P))}\rightarrow\underbrace{\nu_{B_Q}^-{\rm Hom}((P\subset Q),\_)^*}_{\simeq{\rm Hom}(\_,(P\subset Q))}\rightarrow\tau_{B_Q}^{-1}\widehat{\tau Q/P}\rightarrow 0.$$
Applying ${\rm Hom}(\_,F)$ to this sequence and using Yoneda's Lemma, this yields a sequence
$$0\rightarrow{\rm Hom}(\tau_{B_Q}^{-1}\widehat{Q/P},F)\rightarrow F(P\subset Q)\rightarrow F(P=P).$$
The above claimed exactness follows. The theorem is proved.\end{proof}

\begin{cor}\label{cor1f} The following conditions are equivalent for a functor $F$ in $\bmod\,\mathcal{H}_Q^{\rm op}$ and a representation $N$ of $\bmod\, kQ$:
\begin{enumerate}
\item ${\rm res}\, F\simeq N$,
\item There exist exact sequences
$$0\rightarrow G\rightarrow F\rightarrow F'\rightarrow 0\mbox{ and }0\rightarrow F''\rightarrow G\rightarrow \widehat{N}\rightarrow 0$$
such that ${\rm res}\,F'=0={\rm res}\, F''$.
\end{enumerate}
Moreover, a functor $F$ belongs to the essential image of $\Lambda$ if and only if ${\rm Ext}^1(F,\widehat{V})=0={\rm Ext}^1(\widehat{V},F)$ for all $V$.
\end{cor}

\begin{proof} The first claim follows from the previous theorem and exactness of the functor ${\rm res}$. The second claim follows from the previous theorem and \cite[Theorem 5.6]{CFR3}.
\end{proof}

\begin{rem} The second statement of the corollary gives an alternative description of the essential image of the functor $\Lambda$ to \cite[Proposition 6.11]{CFR3}.
\end{rem}

These homological properties allow us to derive the following information on the quotient map $\pi$ and the structure of its fibres:



\begin{cor} The quotient map $\pi$ maps the whole variety $R_{{\widehat{\bf d}}}(B_Q)$ onto $\overline{\mathcal{O}_M}$.
\end{cor}


\begin{proof} Given a functor $F$ of the same dimension vector as $\widehat{M}$ with restriction $N={\rm res}\, F$, we thus have to prove that $\mathcal{O}_N$ belongs to the closure of $\mathcal{O}_M$. By \cite{Bo}, the latter is equicalent to $\dim{\rm Hom}(U,M)\leq\dim {\rm Hom}(U,N)$ holding for all non-projective indecomposables $U$. For each such $U$, use a minimal projective resolution $\xymatrix@1@C=10pt{0\ar[r]& P\ar[r]&Q\ar[r]& U\ar[r]& 0}$ as before and calculate using the definition of the functor $\Lambda$:
$$\dimv_{[U]}F=\dimv_{[U]}\widehat{M}=\dim{\rm Im}({\rm Hom}(Q,M)\rightarrow{\rm Hom}(P,M))=$$
$$=\dim{\rm Hom}(Q,M)-\dim{\rm Hom}(U,M).$$
By the previous theorem, we have $\dimv_{[U]}F\geq\dimv_{[U]}\widehat{N}$, and thus
$$\dim{\rm Hom}(U,M)=\dim{\rm Hom}(Q,M)-\dimv_{[U]}F\leq$$
$$\leq\dim{\rm Hom}(Q,N)-\dimv_{[U]}\widehat{N}=\dim{\rm Hom}(U,N),$$
proving the claim. 
\end{proof}

\begin{rem}
If $\widehat{M}$ and $\widehat{N}$ have the same dimension vector then $M\simeq N$. Indeed, by the lemma above we have $N\in\overline{\mathcal{O}_M}$
and $M\in\overline{\mathcal{O}_N}$.
\end{rem}

One can ask whether this description of the orbit closure $\overline{\mathcal{O}_M}$ as an affine quotient has applications to the study of its geometric properties, in the spirit of the proof \cite{ADK} -- {\it in the case of $Q$ being an equioriented type $A$ quiver} -- of such orbit closures being normal Cohen-Macaulay varieties with at most rational singularities. To this effect, we consider the stratification of the representation variety $R_{\widehat{\bf d}}(B_Q)$ by the inverse images $\pi^{-1}(\mathcal{O}_N)$ of the orbits in $R_{\bf d}(Q)$ under the quotient map $\pi$ and formulate the following:

\begin{conj} For every $N$, we have $\dim\pi^{-1}(\mathcal{O}_N)\leq\dim\overline{\mathcal{O}_{\widehat{M}}}$, with equality holding only for $N=M$.
\end{conj}

We discuss briefly the potential applications of a positive answer to this conjecture. Using the arguments of \cite[2.1,2.2]{BoQ}, the homological Euler form of the algebra $B_Q$ (being of global dimension at most two) can be calculated to give $\langle\widehat{\bf d},\widehat{\bf d}\rangle=\dim{\rm End}(M)$ if $\widehat{\bf d}=\dimv\widehat{M}$. It follows that $\overline{\mathcal{O}_{\widehat{M}}}$ is an irreducible component of $R_{\widehat{\bf d}}(B_Q)$, and that every irreducible component has at least this dimension. Assuming the above conjecture, $\overline{\mathcal{O}_{\widehat{M}}}=R_{\widehat{\bf d}}(B_Q)$ is the only irreducible component, which is locally a complete intersection and thus Cohen-Macaulay. A more refined analysis of the strata $\pi^{-1}(\mathcal{O}_N)$ is expected to prove regularity in codimension one of $R_{\widehat{\bf d}}(B_Q)$. This implies normality, and could thus provide a uniform proof of normality of every orbit closure $\overline{\mathcal{O}_M}$. Note that the first statement of Corollary \ref{cor1f} gives an intrinsic description of the stratum $\pi^{-1}(\mathcal{O}_N)$.



\section{The Hernandez-Leclerc construction and graded Nakajima varieties}\label{HL}
In \cite[Section 9.3]{HL} an algebra $\tilde{\Lambda}_Q$ is introduced. 
Let us collect the main ingredients of the Hernandez-Leclerc construction. Let $Q$ be a Dynkin 
quiver with the set of vertices $I$ of cardinality $n$ and let $\g$ be the corresponding simple Lie algebra.
We denote
by $\al_i$, $i\in I$ simple roots, by $\triangle_+$ and $\triangle_-$ the sets of positive and negative 
roots, and by $W$ the Weyl group of $\g$. Let $s_i\in W$, $i\in I$ 
be the simple reflections, $s_i(\la)=\la - (\la,\al)\al_i$. A Coxeter element $C\in W$ is the product
$s_{i_1}s_{i_2}\dots s_{i_n}$ of simple reflections each showing up exactly once. For a quiver $Q$
we denote by $s_iQ$ a new quiver, which is obtained from $Q$ by reversing all arrows at the vertex $i$. 
The Coxeter element $C$ is said to be adapted to $Q$, if $i_1$ is a source of $Q$, $i_2$ is a source of $s_{i_1}Q$,
$i_{k+1}$ is a source of $s_{i_1}\dots s_{i_k}Q$. For example, if $Q=1\to 2\leftarrow 3\to 4$, then 
$C=s_1s_3s_2s_4$.

A height function $\xi:I\to\bZ$ is a function satisfying $\xi_j=\xi_i-1$ if there is an arrow
$i\to j$ in $Q$. For $Q=1\to 2\leftarrow 3\to 4$ a possible height function is 
$\xi_4=1$, $\xi_3=2$, $\xi_2=1$, $\xi_1=2$. Let us define a set
\[
\widehat I=\{(i,p)\in I\times \bZ:\ p-\xi_i\in 2\bZ\}.
\] 
A crucial role in the whole picture is played by the following bijection 
$\varphi: \widehat I\to\triangle_+\times\bZ$. The bijection $\varphi$ is defined as follows.
First, for a vertex $i\in I$ we denote by $\gamma_i\in\triangle_+$ the sum of all simple roots $\al_j$ 
such that
there is a path form $j$ to $i$ in $Q$ (these are exactly the vertices showing up in the injective envelope
of the simple module attached to the vertex $i$). Now $\varphi$ is defined by the rules:
\begin{itemize}
\item $\varphi(i,\xi_i)=(\gamma_i,0)$;
\item Let $\varphi(i,p)=(\beta,m)$. Then $\varphi(i,p-2)=
\begin{cases}
(C(\beta),m), \text{ if } C(\beta)\in\triangle_+;\\
(-C(\beta),m-1), \text{ if } C(\beta)\in\triangle_-.
\end{cases}$
\end{itemize}           

\begin{rem}
It follows that if $\varphi(i,p)=(\beta,m)$ then 
\[\varphi(i,p+2)=
\begin{cases}
(C^{-1}(\beta),m), \text{ if } C^{-1}(\beta)\in\triangle_+;\\
(-C^{-1}(\beta),m+1), \text{ if } C^{-1}(\beta)\in\triangle_-.
\end{cases}\]
\end{rem}

\begin{example}
Let us give an example. Let $Q=1\to 2\to 3\to 4$ and the height function is fixed as $\xi_1=4$,
$\xi_2=3$, $\xi_3=2$, $\xi_4=1$. The underlying Lie algebra is $\msl_5$ and $\triangle_+$
consists of roots 
$$\al_{i,j}=\al_i+\al_{i+1}+\dots +\al_j,\ 1\le i\le j\le 4.$$
The adapted Coxeter element $C$ is equal to $s_1s_2s_3s_4$.  
Let us draw the set $\widehat I$ as follows:
\medskip
\[
\begin{tabular}{cccc}
(1,6) &  & (3,6) &  \\
 & (2,5) &  & (4,5) \\
(1,4) &  & (3,4) &  \\
 & (2,3) &  & (4,3) \\
(1,2) &  & (3,2) &  \\
 & (2,1) &  & (4,1) \\
(1,0) &  & (3,0) &  \\
 & (2,-1) &  & (4,-1) \\
(1,-2) &  & (3,-2) &  \\
 & (2,-3) &  & (4,-3)  \\
(1,-4) &  & (3,-4) &  
\end{tabular}
\]
Then the corresponding images $\varphi(i,p)$ look as follows:
\medskip
\[
\begin{tabular}{cccc}
$(\al_{1,4},1)$ &               & $(\al_{2,3},1)$ &               \\
              & $(\al_{2,4},1)$ &               & $(\al_{3,3},1)$ \\
$(\al_{1,1},0)$ &               & $(\al_{3,4},1)$ &               \\
              & $(\al_{1,2},0)$ &               & $(\al_{4,4},1)$ \\
$(\al_{2,2},0)$ &               & $(\al_{1,3},0)$ &               \\
              & $(\al_{2,3},0)$ &               & $(\al_{1,4},0)$ \\
$(\al_{3,3},0)$ &               & $(\al_{2,4},0)$ &               \\
              & $(\al_{3,4},0)$ &               & $(\al_{1,1},-1)$ \\
$(\al_{4,4},0)$ &               & $(\al_{1,2},-1)$&                \\
              & $(\al_{1,3},-1)$&               & $(\al_{2,2},-1)$ \\ 
$(\al_{1,4},-1)$ &               & $(\al_{2,3},-1)$&                
\end{tabular}
\]
\end{example}

In \cite{HL} the authors define the following graph $\widetilde \Gamma_Q$. 
The vertices of
$\widetilde \Gamma_Q$ are of two sorts: 
\[
w_j(p), \text{ where } \varphi(j,p)=(\al_i,0) \text{ for some } i\in I
\]
and
\[
v_j(p-1), \text{ where } \varphi(j,p)\in\triangle_+\times \{0\} \text{ and } 
\varphi(j,p-2)\in\triangle_+\times \{0\}.
\]
In particular, the number of vertices $w_j(p)$ is equal to the number of simple roots 
(the number of vertices of $Q$) and 
(as we will see) the number of vertices $v_j(p-1)$ is equal to the number of positive roots minus 
the number of simple roots.
In the example above ($Q$ equioriented type $A_4$) the vertices are
\[
w_1(4), w_1(2), w_1(0), w_1(-2), v_1(3), v_1(1), v_1(-1), v_2(2), v_2(0), v_3(1).
\] 

\begin{rem}
Recall that our $\widehat Q$ from \cite{CFR3} also has two sorts of vertices: vertices $[i]$, $i\in I$ and
vertices $[U]$, that correspond to indecomposable non-projective $U$. In what follows we 
denote the indecomposable representation of $kQ$ corresponding to a positive root $\beta$ by $U_\beta$.
Also for an indecomposable representation $U$ we denote by $\mathbf{dim }\,U\in\triangle_+$ the root
$\sum \dim U_i\al_i$.   
\end{rem}

\begin{lem}
Let $\varphi(j,p)=(\beta,0)\in\triangle_+\times \{0\}$. Then 
$\varphi(j,p-2)\in\triangle_+\times \{0\}$ if and only if $U_\beta$ is non-projective.
\end{lem}
\begin{proof}
The conditions $\varphi(j,p)=(\beta,0)\in\triangle_+\times \{0\}$ and 
$\varphi(j,p-2)\notin\triangle_+\times \{0\}$ mean that $C(\beta)\in\triangle_-$. Hence we are
looking for positive roots $\beta$ such that $C$ maps them to a negative root.
We want to show that these roots correspond to the ones defining projective indecomposable modules.
This is proved in Proposition 4.1  of \cite{KT} (the dual version of that, to be precise). 
\end{proof}

Recall (see \cite{HL}) that the arrows of $\widetilde \Gamma_Q$ are of three types:
\begin{itemize}
\item $a_j(p): w_j(p)\to v_j(p-1)$,
\item $b_j(p): v_j(p)\to w_j(p-1)$,
\item $B_{ij}(p): v_i(p)\to v_j(p-1)$ if there is an arrow $i\to j$ or $j\to i$.
\end{itemize}
\begin{rem}
In $\widehat Q$ we also have three types of arrows.
\end{rem}

\begin{example}
Let $Q=1\to 2\to 3\to 4$ and $\xi_1=4$,
$\xi_2=3$, $\xi_3=2$, $\xi_4=1$ as above. Then the quiver $\widetilde \Gamma_Q$ looks as follows:
$$
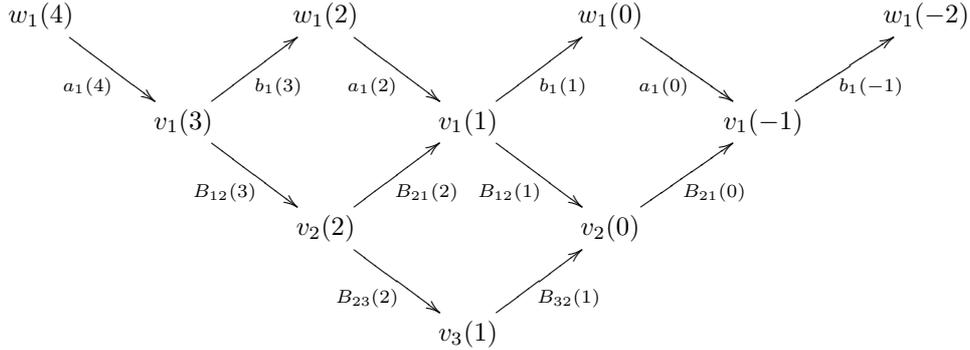
\xymatrix{
w_1(4)\ar[rd]_{a_1(4)} & & w_1(2)\ar[rd]_{a_1(2)} & & w_1(0)\ar[rd]_{a_1(0)} & & w_1(-2)\\
& v_1(3)\ar[ru]_{b_1(3)} \ar[rd]_{B_{12}(3)} & & v_1(1)\ar[ru]_{b_1(1)} \ar[rd]_{B_{12}(1)} & 
& v_1(-1) \ar[ru]_{b_1(-1)}& \\
& & v_2(2)\ar[ru]_{B_{21}(2)} \ar[rd]_{B_{23}(2)} & & v_2(0)\ar[ru]_{B_{21}(0)} & &\\
& & & v_3(1)\ar[ru]_{B_{32}(1)} & & &
}
$$
\end{example}

\begin{example}
Let $Q=1\to 2\leftarrow 3$ and $\xi_1=2$, $\xi_2=1$, $\xi_3=2$. Then $C=s_1s_3s_2$ and the bijection
$\varphi:\widehat I\to \triangle_+\times\bZ $ looks as follows:
\[
\begin{tabular}{ccc}
 (1,4) &  & (3,4) \\
  & (2,3) & \\
 (1,2) &  & (3,2) \\
 & (2,1) &  \\
 (1,0) &  & (3,0) \\
 & (2,-1) &  \\
 (1,-2) &  & (3,-2)\\
 & (2,-3) &  
 \end{tabular}
\qquad \longrightarrow \qquad 
\begin{tabular}{ccc}
 $(\al_{12},1)$ &  & $(\al_{23},1)$ \\
  & $(\al_{22},1)$ & \\
 $(\al_{11},0)$ &  & $(\al_{33},0)$ \\
 & $(\al_{13},0)$ &  \\
 $(\al_{23},0)$ &  & $(\al_{12},0)$ \\
 & $(\al_{22},0)$ &  \\
 $(\al_{33},-1)$ &  & $(\al_{11},-1)$\\
  & $(\al_{13},-1)$ &   
\end{tabular}
\]
Then we have vertices $w_1(2)$, $w_2(-1)$, $w_3(2)$, $v_1(1)$, $v_2(0)$, $v_3(1)$ and
the quiver $\widetilde \Gamma_Q$ looks as follows:
$$
\xymatrix{
w_1(2)\ar[r]^{a_1(2)} & v_1(1)\ar[r]^{B_{21}(1)} & v_2(0)\ar[d]^{b_2(0)} & v_3(1)\ar[l]_{B_{32}(1)} &
w_3(2)\ar[l]_{a_3(2)}\\
& & w_2(-1) & &
}
$$
This picture agrees with the corresponding example in \cite{CFR3}. 
\end{example}

\begin{lem}\label{lem}
$a)$.\ Let $\varphi(j,p)=(\beta,0)$ and $\varphi(j,p-2)=(\gamma,0)$, $\beta,\gamma\in\triangle_+$. 
Then $\tau U_\beta=U_\gamma$.\\
$b)$.\ Assume that the vertices $i,j\in I$ are connected by an arrow $i\to j$ or $j\to i$. 
Let $\varphi(i,p)=(\beta,0)$ and $\varphi(j,p-1)=(\gamma,0)$.
Then there is an irreducible map $U_\gamma\to U_\beta$.
\end{lem}
\begin{proof}
By definition $\gamma=C\beta=C\mathbf{dim }\,U_\beta=\mathbf{dim }\,\tau U_\beta$ 
(for the last equality see e.g. \cite[2.4~(4)]{R} or \cite[lemma~5.8]{ASS}). Since the indecomposable $Q$--representations are uniquely determined by their dimension vector, the claim follows. To prove part (b) let us assume $i\rightarrow j$. Then by definition $U_\beta=I_i$ and $U_\gamma=I_j$ (where $I_k$ denotes the injective envelope of the simple at vertex $k$). In particular, $I_i$ is an indecomposable direct summand of $I_j/\rm{soc }\,I_j$, and hence there is an irreducible map 
$I_j\rightarrow I_i$ (see e.g. \cite[remark IV.4.3 (b)]{ASS}).
\end{proof}

\begin{prop}
The identification 
$$w_j(p)\to [i], \ \varphi(j,p)=(\al_i,0) \text{ and } v_j(p-1)\to [U_\beta],\ \varphi(j,p)=(\beta,0)$$
defines an isomorphism between the quiver $\widetilde \Gamma_Q$ from
\cite{HL} and $\widehat Q$ from \cite{CFR3}.
\end{prop}
\begin{proof}
We need to show that the arrows $a_j(p)$, $b_j(p)$ and $B_{ij}(p)$ are in correspondence with the 
arrows in $\widehat Q$. 

First, let us look at $a_j(p): w_j(p)\to v_j(p-1)$. We know that
$\varphi(j,p)=(\al_i,0)$ (from the definition of $w_j(p)$) and hence $v_j(p-1)$ corresponds 
to $U_{\alpha_i}$ (the simple representation attached to the vertex $i$). We thus get an arrow 
$[i]\to [U_{\alpha_i}]$, which is indeed present in $\widehat Q$.

Second, let us look at $b_j(p): v_j(p)\to w_j(p-1)$. We know that
$\varphi(j,p-1)=(\al_i,0)$. The vertex $v_j(p)$ corresponds to the pair $(j,p+1)\in \widehat I$.
Let $\varphi(j,p+1)=\beta$. From Lemma \ref{lem}, part $a)$, we know that $U_\beta=\tau^{-1} U_{\alpha_i}$.
We thus obtain an arrow $[\tau^{-1} U_{\alpha_i}]\to [i]$, which is indeed present in $\widehat Q$.   

Finally we consider the arrows $B_{i,j}(p): v_i(p)\to v_j(p-1)$, where $i$ and $j$ are connected by
an arrow ($i\to j$ or $j\to i$). The existence of the corresponding arrows in $\widehat Q$ follows from
Lemma \ref{lem}, part $b)$.
\end{proof}

Finally, the algebra $\tilde{\Lambda}_Q$ is defined in \cite[Section 9.3]{HL} as the path algebra of the quiver $\widetilde \Gamma_Q$ subject to the relations
\begin{equation}\label{HLrelation} a_i(p-1)b_i(p)=\sum_{j-i}\epsilon(i,j)B_{ji}(p-1)B_{ij}(p),
\end{equation}
for all vertices $i$ and all $p$, where the sum ranges over all vertices $j$ adjacent to $i$, and $\epsilon(i,j)$ is an appropriate sign. For our purposes, a modified form of these relations (without the signs) will be more suitable:

\begin{lem} The algebra $\tilde{\Lambda}_Q$ is isomorphic to the path algebra of the quiver $\widetilde\Gamma_Q$ subject to the relations
\begin{equation}\label{HLrelationwosign} a_i(p-1)b_i(p)=\sum_{j-i}B_{ji}(p-1)B_{ij}(p),
\end{equation}
for all vertices $i$ and all $p$, where the sum ranges over all vertices $j$ adjacent to $i$.
\end{lem}

\begin{proof} We twist the arrow $B_{i,j}(p)$ by a sign $\epsilon(i,j,p)$. Then the sign $\epsilon(i,j)$ in (\ref{HLrelation}) vanishes and gives the desired relation (\ref{HLrelationwosign}) if and only if
$\epsilon(j,i,p-1)\epsilon(i,j,p)\epsilon(i,j)=1$
for all $i,j,p$, which can be rewritten as
$$\epsilon(j,i,p-1)=\epsilon(i,j)\epsilon(i,j,p).$$
This gives a way to define the correct signs inductively: for every pair of adjacent vertices $i,j$ in $Q$, we take the maximal $p$ such that $B_{i,j}(p)$ exists and choose $\epsilon(i,j,p)$ arbitrarily. Then for all smaller $q$, the above relation tells us how to define the sign $\epsilon(i,j,q)$. 
\end{proof}

\begin{thm} 
The algebra $\tilde{\Lambda}_Q$ is isomorphic to the algebra $B_Q$.
\end{thm}
\begin{proof} We compute defining relations for the algebra $B_Q$. Since $B_Q$ is of global dimension at most two, 
we can localize the relations by computing ${\rm Ext}^2$ between simple objects by \cite[Section 1]{BoQ}, which in the language 
of contravariant functors on $\mathcal{H}_Q$ are the $S_{P_U\subset Q_U}$ for $U$ a non-projective 
indecomposable over $kQ$, and the $S_{P_i=P_i}$ for the vertices $i$ of $Q$. It is shown in \cite{CFR3} that 
the $S_{P_i=P_i}$ have projective and injective dimension at most one, thus it suffices to compute 
${\rm Ext}^2(S_{P_U\subset Q_U},S_{P_V\subset Q_V})$ for two non-projective indecomposables $U$ and $V$ over 
$kQ$. Moreover, it is shown in \cite{CFR3} that ${\rm Ext}^*(S_{P_U\subset Q_U},F)$ can be computed as the 
homology of the complex
$$F(P_U\subset Q_U)\rightarrow F(P_B\subset Q_B)\rightarrow F(P_{\tau U}\subset Q_{\tau U}),$$
where $0\rightarrow\tau U\rightarrow B\rightarrow U\rightarrow 0$ is the Auslander-Reiten sequence ending in $U$, 
and $P_B=P_U\oplus P_{\tau U}$, $Q_B=Q_U\oplus Q_{\tau U}$. Thus, if
${\rm Ext}^2(S_{P_U\subset Q_U},S_{P_V\subset Q_V})$ is non-zero, we have $V=\tau U$, in which case
${\rm Ext}^2(S_{P_U\subset Q_U},S_{P_V\subset Q_V})$ is one-dimensional. Thus to define the algebra $B_Q$ 
as a quotient of $k\widehat{Q}$, it suffices to determine a single relation involving paths from the vertex $[U]$
to the vertex $[\tau U]$. To do this, we compare morphisms in $\mathcal{H}_Q$ and in $\bmod\, kQ$. By \cite{CFR3}, 
we have an isomorphism 
$$\underline{\rm Hom}_{\mathcal{H}_Q}((P\subset Q),(R\subset S))\simeq{\rm Hom}_{kQ}(Q/P,S/R),$$ 
where $\underline{\rm Hom}_{\mathcal{H}_Q}((P\subset Q),(R\subset S))$ is defined as the quotient of 
${\rm Hom}_{\mathcal{H}_Q}((P\subset Q),(R\subset S))$ by the subspace generated by morphisms factoring through 
an object of $\mathcal{H}_Q$ of the form $(T=T)$. In particular, we have 
${\rm Hom}_{\mathcal{H}_Q}(P_U\subset Q_U,P_V\subset Q_V)\simeq{\rm Hom}(U,V)$ if there exists an irreducible 
map from $U$ to $V$, and 
${\rm Hom}_{\mathcal{H}_Q}(P_{\tau U}\subset Q_{\tau U},P_U\subset Q_U)\simeq{\rm Hom}(\tau U,U)$ 
if there is no vertex $[j]$ between $[\tau U]$ and $[U]$ in $\widehat{Q}$. First let us assume that this is 
the case. Then the relation in $B_Q$ between $[\tau U]$ and $[U]$ is the mesh relation, that is, the sum over 
all paths from $[\tau U]$ to $[U]$ is zero. Otherwise, let us assume that there is a vertex $[j]$ between 
$[\tau U]$ and $[U]$. Then, by definition of $\widehat{Q}$, we have $U=S_i$ and $\tau U=\tau S_i=S_j$ for 
vertices $i$ and $j$ of $Q$. In particular, we have an Auslander-Reiten sequence 
$0\rightarrow S_j\rightarrow B\rightarrow S_i\rightarrow 0$ in $\bmod\, kQ$, so that there exists an 
arrow $i\rightarrow j$ in $Q$ and $B$ is the two-dimensional indecomposable supported on $i$ and $j$. 
A direct computation using the projective resolution 
$0\rightarrow\bigoplus_{j\rightarrow k}P_k\rightarrow P_j\rightarrow S_j\rightarrow 0$ of $S_j$ 
(and similarly for $S_i$) shows that in this case we have a commutativity relation for the commutative 
square
$$\begin{array}{ccccc}&&[B]&&\\ &\nearrow&&\searrow&\\ {[S_i]}&&&&[\tau S_i=S_j]\\ &\searrow&&\nearrow&\\ &&[j].&&\end{array}$$
We conclude that the defining relations of $B_Q$ coincide with the defining relations \eqref{HLrelationwosign} of $\tilde{\Lambda}_Q$ under the identification between the quivers $\widetilde \Gamma_Q$ and $\widehat Q$.
\end{proof}

\section*{Acknowlegments}
The authors would like to thank K. Bongartz, B. Leclerc, B. Keller and S. Scherotzke  for helpful discussions.

The work of Evgeny Feigin was partially supported
by the Russian President Grant MK-3312.2012.1, by the Dynasty Foundation,
by the AG Laboratory HSE, RF government grant, ag. 11.G34.31.0023, by the RFBR grants
12-01-00070, 12-01-00944, 12-01-33101 and by the Russian Ministry of Education and Science under the
grant 2012-1.1-12-000-1011-016.
This study comprises research fundings from the `Representation Theory
in Geometry and in Mathematical Physics' carried out within The
National Research University Higher School of Economics' Academic Fund Program
in 2012, grant No 12-05-0014.
This study was carried out within the National Research University Higher School of Economics
Academic Fund Program in 2012-2013, research grant No. 11-01-0017.

\bibliographystyle{amsplain}

\end{document}